\documentclass[a4paper,12pt]{amsart}
\usepackage{amsmath,amsthm}
\usepackage{amsfonts,amsmath,amsthm, amssymb, mathrsfs}
\usepackage{float}
\usepackage{euscript,eufrak,verbatim}
\usepackage{graphicx}
\usepackage[usenames]{color}
\usepackage[colorlinks,linkcolor=red,anchorcolor=blue,citecolor=blue]{hyperref}
\usepackage{amsmath}
\usepackage{amsthm}
\usepackage[all]{xy}
\usepackage{graphicx} 
\usepackage{amssymb} 
\usepackage{enumerate}
\usepackage{enumitem}

\newtheorem{thm}{Theorem}[section]
\newtheorem{prop}[thm]{Proposition}
\newtheorem{df}[thm]{Definition}
\newtheorem{lem}[thm]{Lemma}

\newtheorem{ex}[thm]{Example}

\newtheorem{rem}[thm]{Remark}

\def\logf{\log\frac{1}{\epsilon}}
\makeatletter 
\@addtoreset{equation}{section}
\numberwithin{equation}{section}

\title{Variational principle for neutralized  Bowen topological entropy}

\author{Rui Yang, Ercai Chen and Xiaoyao Zhou*
}
\address
{1.School of Mathematical Sciences and Institute of Mathematics, Nanjing Normal University, Nanjing 210023, Jiangsu, P.R.China}
\email{zkyangrui2015@163.com}
\email{ecchen@njnu.edu.cn}
\email{zhouxiaoyaodeyouxian@126.com}
\date{}

\begin{document}

\maketitle

\renewcommand{\thefootnote}{}
\footnote{*corresponding author}

\begin{abstract}
Ovadia and Rodriguez-Hertz defined   neutralized Bowen open ball as
$$B_n(x,e^{-n\epsilon})=\{y\in X: d(T^jx, T^jy)<e^{-n\epsilon}, \forall 0\leq j\leq n-1\}.$$
We introduce the notion of  neutralized  Bowen topological entropy of subsets  by neutralized Bowen open ball, and establish variational principles  for  neutralized  Bowen topological entropy of compact subsets in terms of  neutralized  Brin-Katok local entropy and neutralized  Katok's  entropy.
\end{abstract}

\section{Introduction}

By a pair  $(X,T)$ we mean a topological dynamical system (TDS for short), where $X$ is a compact  metric space  with  a metric $d$ and   $T$ is a continuous self-map  on $X$.
By $\mathcal{M}(X)$  we   denote  the set of all Borel probability measures on $X$. 

The topological entropy introduced by  Adler,  Konheim  and   McAndrew \cite{akm65} is  a vital topological  invariant to capture the topological complexity of dynamical systems. Later, Bowen \cite{b71} gave equivalent definitions  for topological entropy by  spanning sets and separated sets.  From the viewpoint of qualitative theory, entropy reflects the  exponential growth rate  of   the number of orbits segments up to a fixed   accuracy  that  can be  distinguished     as the time evolves. In 1973,  Bowen \cite{b73} in his profound paper  further revealed  that entropy   essentially can  be  viewed as  the analogues of dimension from the viewpoint of dimension theory.  Until now, topological entropy is an  important tool to understand the dynamics of  systems, which  plays a key role in  topological dynamics, ergodic theory, multifractal analysis, mean dimension theory and  other fields of dynamical systems.

 The present paper focus on  the interplay between ergodic theory and topological entropy. It is well-known that the classical variational principle \cite{w82} for  topological entropy  provides a  bridge between ergodic theory and topological dynamics. Namely,
$$h_{top}(T,X)=\sup_{\mu \in \mathcal{M}(X,T)}h_{\mu}(T),$$
where $h_{top}(T,X)$ denotes the topological entropy of $X$, and $h_{\mu}(T)$ is the Kolmogorov-Sinai  entropy of  invariant measure $\mu$. In 2012, Feng and Huang  \cite{fh12} showed that  if  the  compact subset is not invariant, then  such a  variational principle still exists for Bowen topological entropy of compact subsets in terms of  lower Brin-Katok local entropy. Precisely,

$$h_{top}^B(T,K)=\sup\{\underline{h}_{\mu}^{BK}(T): \mu \in \mathcal{M}(X), \mu(K)=1\},$$
where $h_{top}^B(T,K)$  denotes the Bowen topological entropy of compact subset $K$, and  $\underline{h}_{\mu}^{BK}(T)$ is the  lower Brin-Katok  local entropy of $\mu$. The variational principles  on subsets of non-autonomous systems \cite{xz18},   amenable groups \cite{zc16,hlz20}, free semi-groups \cite{zc21,xm22}, fixed-point free flows \cite{dfq17}, partially hyperbolic systems \cite{tw22,ljz22},  were exhibited in the setting of  different  types of  dynamical systems.  Ovadia and Rodriguez-Hertz \cite{orh23} defined the  neutralized Bowen open ball as
$$B_n(x,e^{-n\epsilon})=\{y\in X: d(T^jx, T^jy)<e^{-n\epsilon}, \forall 0\leq j\leq n-1\}.$$
As the Brin-Katok entropy formula shown,  Bowen balls  $\{B_n(x,\epsilon)\}_{n\geq 0}$ allow us to  control their measure and  image under the dynamics for  iterations,  while it may  possess   complicated geometric shape  in a central direction.  Using  neutralized Bowen open balls,  Ovadia and Rodriguez-Hertz   introduced  neutralized Brin-Katok local entropy to   estimate the asymptotic measure of sets with a distinctive geometric shape by  neutralizing the sub-exponential effects, which turns out that such open  balls  have more advantages for describing the  neighborhood with a local linearization of the dynamics.

Notice that the  classical Bowen topological entropy \cite{b73,p97,fh12} is  only defined by the  Bowen balls whose radius is   fixed.   This  invokes a natural question is  whether  the aforementioned variational  principle  still holds  for neutralized  Bowen topological entropy defined by  neutralized Bowen  balls.  To this end, we introduce two new notions, called  neutralized  Brin-Katok local entropy and neutralized  Katok's  entropy of Borel probability measure,  as the  counterpart of the role  of measure-theoretic entropy played in ergodic theory. Borrowing   some ideas from geometric measure theory and  utilizing  dimensional approach,  we  obtain an analogous variational principle. Surprisingly,  the form of variational principle  is different from the  classical variational principles for topological entropy  in terms of Kolmogorov-Sinai entropy \cite{w82},  and Bowen topological entropy of compact subsets in terms of lower Brin-Katok local entropy \cite{fh12},  which is  more close to the Lindenstrauss-Tsukamoto's  variational principle \cite[Theorems 6 and 9]{lt18} for metric mean dimension in terms of rate-distortion functions. Precisely,  we state it as follows. 
\begin{thm}\label{thm 1.1}
Let $(X,T)$ be a TDS   and  $K$ be a non-empty compact subset of $X$.  Then 

\begin{align*}
h_{top}^{\widetilde{B}}(T,K)&=\lim_{\epsilon \to 0}\sup\{\underline{h}_{\mu}^{\widetilde{BK}}(T,\epsilon):\mu\in \mathcal{M}(X),\mu(K)=1\},\\
&=\lim_{\epsilon \to 0}\sup\{h_{\mu}^{\widetilde{K}}(T,\epsilon):\mu\in \mathcal{M}(X),\mu(K)=1\},
\end{align*}
where  $h_{top}^{\widetilde{B}}(T,K)$  denotes   neutralized  Bowen topological entropy of $K$, and  $\underline{h}_{\mu}^{\widetilde{BK}}(T,\epsilon), h_{\mu}^{\widetilde{K}}(T,\epsilon)$  are    lower  neutralized Brin-Katok local entropy of $\mu$,   neutralized Katok entropy of $\mu$, respectively.
\end{thm}

The rest of this paper is organized as follows. In section 2,  we introduce the notions of  neutralized  Bowen topological entropy of subsets, neutralized  Brin-Katok local entropy, and neutralized  Katok's  entropy of Borel probability measures.   In  section 3, we give  the proof of Theorem  1.1.

\section{Preliminary}
In this section, we introduce the notions of  neutralized  Bowen topological entropy of subsets by using two different approaches,  and define neutralized  Brin-Katok local entropy and neutralized  Katok's  entropy of Borel probability measures. 
\subsection{Neutralized topological entropies of subsets}

Given   $n\in \mathbb{N}$, $x,y \in X$, the $n$-th Bowen metric $d_n$  on $X$ is defined by  $$d_n(x,y):=\max_{0\leq j\leq n-1}\limits d(T^{j}(x),T^j(y)).$$ Then  \emph{Bowen open  ball} of radius $\epsilon$ and  order $n$ in the metric $d_n$ around $x$ is   given by 
$$B_n(x,\epsilon)=\{y\in X: d_n(x,y)<\epsilon\}.$$

Recall that in  \cite{orh23}  the authors  defined  neutralized Bowen open ball by replacing  the radius $\epsilon$ in  usual Bowen ball  by $e^{-n\epsilon}$. Precisely,   the  \emph{neutralized Bowen open  ball} of radius $\epsilon$ and  order $n$ in the metric $d_n$ around $x$  is given by  
$$B_n(x, e^{-n\epsilon})=\{y\in X: d_n(x,y)<e^{-n\epsilon}\}.$$

Following the idea of \cite{b73,p97,fh12},  using neutralized Bowen balls we define the so-called \emph{neutralized Bowen topological entropy} of subset  which is not necessarily compact or $T$-invariant.

Let  $Z\subset X$ be a non-empty subset, $\epsilon>0$, $N\in \mathbb{N}$ and $s \geq 0$.  
Put
$$M_{N,\epsilon}^s(Z)=\inf\sum_{i\in I}\limits  e^{-n_i s},$$
where the infimum  is taken over all  finite or countable covers $\{B_{n_i}(x_i,e^{-n_i\epsilon})\}_{i\in I}$ of $Z$ with $n_i \geq N,x_i \in X.$

Clearly,  the limit $M_{\epsilon}^s(Z)=\lim\limits_{N\to \infty}M_{N,\epsilon}^s(Z)$ exists since  $M_{N,\epsilon}^s(Z)$ is non-decreasing when $N$ increases.

\begin{prop}\label{prop 2.1}
If  there exists $s$  such  that $M_{\epsilon}^s(Z)<\infty$, then $M_{\epsilon}^t(Z)=0$ for any $t>s$; If $M_{\epsilon}^s(Z)>0$ for some $s$, then $M_{\epsilon}^t(Z)=\infty$  for any $t <s$.	
\end{prop}

\begin{proof}
It suffices to show the first statement.  Assume that $c:=M_{\epsilon}^s(Z)<\infty$.  For sufficiently large  $N$,  there is  a  finite or countable covers $\{B_{n_i}(x_i,e^{-n_i\epsilon})\}_{i\in I}$ of $Z$ with $n_i \geq N,x_i \in X$ so  that  $\sum_{i\in I}e^{-n_i s}<c+1$. Fix $t >s$. Then
$$M_{N,\epsilon}^t(Z)\leq  \sum_{i\in I}e^{-n_i t}\leq e^{(s-t)N} \sum_{i\in I} e^{-n_i s},$$
which implies that $M_{\epsilon}^t(Z)=0$ by letting $N \to \infty$.
\end{proof}

By Proposition \ref{prop 2.1},  the quantity  $M_{\epsilon}^s(Z)$ has a critical value  of parameter $s$  jumping from $\infty$ to $0$, which  is defined by 
\begin{align*}
	M_{\epsilon}(Z):=\inf\{s:M_{\epsilon}^s(Z)=0\}
	=\sup\{s:M_{\epsilon}^s(Z)=\infty\}.
\end{align*}
\begin{df}
Given non-empty subset  $Z\subset X$, the  neutralized  Bowen topological entropy of $T$  on the set $Z$ is  defined by 
\begin{align*}
h_{top}^{\widetilde{B}}(T,Z):=\lim_{\epsilon \to 0}M_{\epsilon}(Z)=\inf_{\epsilon >0} M_{\epsilon}(Z).
\end{align*}

\end{df}

The next proposition presents  some basic properties related to  neutralized  Bowen topological entropy.

\begin{prop} \label{prop 2.2}
\begin{enumerate}
\item For every $Z\subset X$, the value of   $h_{top}^{\widetilde{B}}(T,Z)$   is independent of the choice of metrics on $X$.
\item If $Z_1\subset Z_2 \subset X$, then $h_{top}^{\widetilde{B}}(T,Z_1)\leq h_{top}^{\widetilde{B}}(T,Z_2)$.\\
\item If $Z$ is a countable union  of $Z_i$, then  
$M_{\epsilon}(Z)=\sup_{i\geq 1}\limits M_{\epsilon}(Z_i)$, and  $h_{top}^{\widetilde{B}}(T,Z)\geq \sup_{i\geq 1} \limits
h_{top}^{\widetilde{B}}(T,Z_i)$; If  $Z$ is a finite union of $Z_i$, $i=1,...,N$, then $h_{top}^{\widetilde{B}}(T,Z)=\max_{1 \leq i\le N} \limits
h_{top}^{\widetilde{B}}(T,Z_i)$.
\end{enumerate}
\end{prop}

\begin{proof}
It suffices to show (1) since the rest of statements are direct
consequence of the definition.  Let $d_1, d_2$ be two compatible metrics on $X$. Then for  every $\epsilon^{'} >0$ there exists $\delta^{'} >0$ such that for all $x,y \in X$ with  $d_1(x,y)<\delta^{'}$, one has $d_2(x,y) <\epsilon^{'}$. Now, fix $\epsilon >0 $ and $0<\delta <\epsilon$. For every $N\geq 1$, there exists $n\geq N$ (depending on $\epsilon$, $\delta$ and $N$) so that   for all $x,y \in X$ with  $d_1(x,y)<e^{-n\delta}$, one has $d_2(x,y) <e^{-N\epsilon}$. Using this fact, we have  $M_{N,\epsilon,d_2}^s(Z)\leq M_{n,\delta,d_1}^s(Z)\leq  M_{\delta,d_1}^s(Z)$  and hence $M_{\epsilon,d_2}^s(Z)\leq   M_{\delta,d_1}^s(Z)$. This  implies that  $M_{\epsilon,d_2}(Z)\leq   M_{\delta,d_1}(Z)$.   Letting $\epsilon \to 0$  one has  $h_{top}^{\widetilde{B}}(T,Z,d_2)\leq h_{top}^{\widetilde{B}}(T,Z,d_1)$. Exchanging the role of $d_1$ and $d_2$  we get the converse inequality.
\end{proof}

Analogous to the  definition of classical topological entropy,  we  also  define  the  \emph{neutralized  topological entropy of  the subset $Z$ of $X$}    by spanning sets and separated sets.

Given $\epsilon>0$, $n\in \mathbb{N}$ and non-empty subset $Z$, a set $E\subset X$ is  \emph{an $(n,\epsilon)$-spanning set of $Z$} if  for any $x \in Z$, there  exists  $y\in E$ such that $d_n(x,y)< e^{-n\epsilon}.$ The smallest cardinality of $(n,\epsilon)$-spanning sets of $Z$  is denoted by $r_n(Z, \epsilon)$.   A set $F\subset Z$ is  \emph{an $(n,\epsilon)$-separated set of $Z$} if  for any $x, y \in F$ with $x\not=y$, one has  $d_n(x,y)\geq e^{-n\epsilon}.$ The largest  cardinality of $(n,\epsilon)$-separated sets of $Z$  is denoted by $s_n(Z, \epsilon)$.  Put 
\begin{align*}
r(Z,\epsilon)&=\limsup_{n\to \infty} \frac{1}{n} \log r_n(Z,\epsilon),\\
s(Z,\epsilon)&=\limsup_{n\to \infty} \frac{1}{n} \log s_n(Z,\epsilon).
\end{align*}

\begin{df}
We define the neutralized    topological entropy of  $Z$ w.r.t to $T$ as
$${\widetilde h}_{top}(T,Z)=\lim\limits_{\epsilon \to 0}r(Z,\epsilon)=\inf_{\epsilon>0}r(Z,\epsilon).$$
\end{df}
The  following proposition  suggests that neutralized    topological entropy of  $Z$   can be equivalently given by separated sets.
\begin{prop}
Let $(X,T)$ be a TDS. Then  for any non-empty subset $Z\subset X$,
$${\widetilde h}_{top}(T,Z)=\lim\limits_{\epsilon \to 0}s(Z,\epsilon).$$
\end{prop}

\begin{proof}
The inequality $r(Z,\epsilon)\leq s(Z,\epsilon)$ follows by using the fact  of an $(n,\epsilon)$-separated set of $Z$ with largest cardinality  is  an $(n,\epsilon)$-spanning set of $Z$.

On the other hand, fix $\epsilon>0$ and then choose $n_0$ such that  $2e^{-\frac{\epsilon}{2}n}<1$ for all $n \geq n_0$.   Fix $n \geq n_0$.
Let  $E$ be an $(n,\epsilon)$-spanning set of $Z$, and $F$ be an $(n,\frac{\epsilon}{2})$-separated set of $Z$. Define a map $f: F\rightarrow E$ by choosing  a fixed $f(x) \in E$  so that $d_n(x, f(x))<e^{-n\epsilon}$ for each $x\in F$. Then $f$ is injective.  Otherwise,  say $x,y \in F$ with $x\not=y$ such that $f(x)=f(y)=z\in E$. Then $d_n(x,y)<2e^{-n\epsilon}=2e^{-\frac{\epsilon}{2}n}\cdot e^{-\frac{\epsilon}{2}n}<e^{-\frac{\epsilon}{2}n}$, a  contradiction due to the choice of $x,y$. Therefore, $ s_n(Z, \frac{\epsilon}{2})\leq r_n(Z, \epsilon)$, which  gives  us $s(X,\frac{\epsilon}{2})\leq r(Z,\epsilon)$.
\end{proof}

\begin{prop}\label{prop 2.6}
Let $(X,T)$ be a TDS. Then  for any non-empty subset $Z\subset X$,
$${\widetilde h}_{top}^B(T,Z)\leq {\widetilde h}_{top}(T,Z).$$
\end{prop}
\begin{proof}
Suppose that  ${\widetilde h}_{top}(T,Z)<\infty$.  Otherwise, there is nothing left to prove.  Let  $s>{\widetilde h}_{top}(T,Z)$. Then there is $\epsilon>0$ satisfying   for every sufficiently large $n$, there is an $(n,\epsilon)$-spanning set $E$ of $Z$ so that  $\#E<e^{ns}.$ 
Notice that $Z\subset \cup_{x\in E}B_n(x,e^{-n\epsilon})$. Then
$$M_{n,\epsilon}^s(Z)\leq \sum_{x\in E} e^{-ns}<1,$$
which implies  that  ${\widetilde h}_{top}^B(T,Z)\leq M_{\epsilon}(Z)\leq s$.  Letting $s \to {\widetilde h}_{top}(T,Z)$, we get ${\widetilde h}_{top}^B(T,Z)\leq {\widetilde h}_{top}(T,Z)$.
\end{proof}

\begin{rem}
Denote by  $h_{top}^{{B}}(T,Z)$, $h_{top}(T,Z)$   the  Bowen topological entropy of  $T$  on the set $Z$,  (upper capacity)topological entropy of  $T$  on the set $Z$ defined by Bowen open balls  $\{B_{n_i}(x_i,\epsilon)\}_{i\in I}$.  Since  for given $\epsilon >0$, $B_n(x, e^{-n\epsilon}) \subset B_n(x, \epsilon)$ for sufficiently large $n$, it is  easy to see that   $h_{top}^{{B}}(T,Z)\leq h_{top}^{\widetilde{B}}(T,Z),$ and  $h_{top}(T,Z)\leq {\widetilde h}_{top}(T,Z)$. 
\end{rem}
Compared with the classical entropy-like quantities, the  neutralized  entropies are  more difficult  to be computed  since we  have  to  consider the varied radius in each step. The following example   only provides  the estimation of  the lower and upper bounds of  neutralized  entropy  for  symbolic systems. Moreover, we show there exists a set so that  the inequality is strict  for Proposition \ref{prop 2.6}.
\begin{ex}\label{ex 2.7}
Let   $\mathcal{A}=\{0,...,N-1\}, N\geq 3$,  be  $N$ symbols equipped with discrete topology. Endowed  $\Sigma_N=\{0,...,N-1\}^{\mathbb{N}}$ with product topology, which is compatible  with  the metric   $$d(x,y)=N^{-\min\{n\geq 0:    x_n\not=y_n\}}$$  for  $x,y \in \Sigma_N$  with  $x\not=y$,  and $d(x,y)=0$ for $x=y$. The one-sided full shift  $(\Sigma_N,\sigma)$  is given by  the left shift $\sigma(x)=(x_{n+1})_{n\geq 0}$, where  $x=(x_{n})_{n\geq 0}$.  Then  for any   non-empty subset $Z\subset \Sigma_N$, one has

$$h_{top}^B(T,Z)\leq {\widetilde h}_{top}^B(T,Z)\leq {\widetilde h}_{top}(T,Z)\leq \overline{\rm dim}_B(Z)\cdot \log N,$$
where  $\overline{\rm dim}_B(Z)$ is the box-counting dimension of $Z$.

In particular, if  $Z$ is $\sigma$-invariant compact subset, then the equalities hold.  Consequently,
$$\overline{\rm dim}_B(Z)=\frac{{\widetilde h}_{top}^B(T,Z)}{\log N}.$$
\end{ex}

\begin{proof}
 By definitions and   Proposition \ref{prop 2.6}, it suffices to show  ${\widetilde h}_{top}(T,Z)\leq \overline{\rm dim}_B(Z)\cdot \log N$. Recall that $\overline{\rm dim}_B(Z)=\limsup_{\epsilon \to 0}\frac{\log R_{\epsilon}(Z,d)}{\logf}$, where  $R_{\epsilon}(Z,d)$ denotes the smallest cardinality of  open balls  formed by $B_d(x,\epsilon)$ covering $Z$. It is  readily to  show that  
$${\widetilde h}_{top}(T,X)=\lim\limits_{k \to \infty}\limsup_{m\to \infty} \frac{1}{mk} \log r_{mk}(X,\frac{1}{k}).$$
Fix $k\geq 1$ and then  choose subsequence $m_j$ which converges to $\infty$ as $j \to \infty$ so that
$$\limsup_{m\to \infty} \frac{1}{mk} \log r_{mk}(X,\frac{1}{k})=\limsup_{j\to \infty} \frac{1}{m_jk} \log r_{m_jk}(X,\frac{1}{k}).$$

Without loss of generality, we may assume that  the limit $$s:=\lim_{j\to \infty}\frac{\log R_{N^{-m_j(k+1)}}(Z,d)}{\log N^{m_j(k+1)}}$$
exists.
Let $\alpha >s$. Then for sufficiently large $j$,  there exists open balls $\{B_d(x_i,N^{-m_j(k+1)})\}_{i\in I}$ satisfying that
$\#I<N^{\alpha m_j(k+1)}$ and $$Z\subset \cup_{i \in I} B_d(x_i,N^{-m_j(k+1)}).$$

Since  $$B_d(x_i,N^{-m_j(k+1)}) \subset B_{m_jk}(x_i,e^{-m_j})=B_{m_jk}(x_i,e^{-m_jk\cdot \frac{1}{k}})$$ for each $i \in I$, then $ r_{m_jk}(Z,\frac{1}{k})\leq  N^{\alpha m_j(k+1)}$.   So $r(Z,\frac{1}{k})\leq \alpha (1+\frac{1}{k})\log N$.  Letting $\alpha \to s$ and   $k \to \infty$, we have ${\widetilde h}_{top}^B(T,Z)\leq \overline{\rm dim}_B(Z)\log N.$

If $Z$ is $\sigma$-invariant compact subset, then   by \cite{kp91} $\overline{\rm dim}_B(Z)=\frac{h_{top}(T,Z)}{\log N}$. By \cite[Proposition 1]{b73}, $h_{top}(T,Z)=h_{top}^B(T,Z)$.
This implies that   the equalities hold.

We define $$Z_k=\{x=(x_n)\in \Sigma_N:  x_n=0  \text{~for~} n\geq k\},$$
and $Z=\cup_{k \geq 1}Z_k$. Then $Z$ is dense in $\Sigma_N$. Fix $\epsilon >0$. One has $M_{\epsilon}(Z)=\sup_{k \geq 1}\limits M_{\epsilon}(Z_k)=0$ since $Z_k$ is a finite subset. Hence ${\widetilde h}_{top}^B(T,Z)=0$. Notice that
$${\widetilde h}_{top}(T,Z)={\widetilde h}_{top}(T,\overline{Z})={\widetilde h}_{top}(T,X)=\log N.$$
This shows $${\widetilde h}_{top}^B(T,Z)< {\widetilde h}_{top}(T,Z).$$
\end{proof}

As the role  of Kolmogorov-Sinai entropy played in the classical variational principle for  topological entropy \cite{w82}, we also need  the  lower  neutralized Brin-Katok local   entropy   and   the   neutralized Katok entropy of Borel probability measures  to establish variational principle for  neutralized  Bowen topological entropy on compact subsets.

\subsection{ Lower  neutralized Brin-Katok  local entropy } 
Likewise,  one can employ the neutralized  Bowen ball to  define the  local  lower  neutralized Brin-Katok entropy by following the idea of \cite{bk83,fh12}. 

\begin{prop}
Let $\mu \in \mathcal{{M}}(X)$ and $\epsilon >0$. Then  the  function $f_n(x,\epsilon):=\mu (B_n(x, e^{-n\epsilon})$  defined on $X$ is Borel measurable.
\end{prop}
\begin{proof}
Set $r=e^{-n\epsilon}>0$ and put $E:=\{x\in X: \mu (B_n(x, e^{-n\epsilon}) >s \}$.  Fix $x\in E$ and notice that $B_n(x,r)=\cup _{j\geq 1}B_n(x,r-\frac{1}{n})$. Then  by the continuity of measure  of $\mu$, $\lim\limits_{n \to \infty}\mu(B_n(x,r-\frac{1}{n}))>s$. So  one can   choose $0<r_0<r$ so that $\mu(B_n(x,r_0))>s$.  Let $U_x:=\{y\in X:d_n(x,y)<r-r_0\}$ be an open neighborhood of $x$.  This immediately implies that $\mu(B_n(y,r))\geq \mu(B_n(x,r_0))>s$ for every $y\in U_x$. This yields that $E$ is an open set.  Hence,  $f_n(x,\epsilon)$  is Borel measurable.
\end{proof}

\begin{df}\label{def 2.5}
Given $\mu \in \mathcal{{M}}(X)$ and $\epsilon >0$,  we put
\begin{align*}
	\underline{h}_{\mu}^{\widetilde{BK}}(T, \epsilon):=\int \liminf_{n\to \infty}-\frac{\log \mu (B_n(x, e^{-n\epsilon}))}{n}d\mu.
\end{align*}
We define the lower  neutralized Brin-Katok local  entropy of $\mu$ as  $$\underline{h}_{\mu}^{\widetilde{BK}}(T)=\lim_{\epsilon\to 0}\underline{h}_{\mu}^{\widetilde{BK}}(T, \epsilon).$$
\end{df}

\begin{rem}


Replacing the ball $B_n(x, e^{-n\epsilon})$ with $ B_n(x,\epsilon)$ in Definition \ref{def 2.5}, Feng and Huang   formulated the  local  lower   Brin-Katok entropy  of Borel probability measure $\mu$  as follows:
\begin{align*}
h_{\mu}^{BK}(T)&=\int  \lim\limits_{\epsilon \to 0} \liminf_{n\to \infty}-\frac{\log \mu (B_n(x, \epsilon)}{n}d\mu,\\
&=\lim\limits_{\epsilon \to 0}\int  \liminf_{n\to \infty}-\frac{\log \mu (B_n(x, \epsilon)}{n}d\mu,
\end{align*}
where   we used monotone convergence theorem for the second equality.   It is readily to see that  $h_{\mu}^{BK}(T)\leq \underline{h}_{\mu}^{\widetilde{BK}}(T)$.
\end{rem}

\subsection{ Neutralized  Katok's entropy}
The similar procedure  can be applied to the  classical Katok's entropy defined by spanning set \cite{k80}. However,  to pursue a variational principle we need to define  the   neutralized  Katok's entropy of Borel probability measure by Carath\'eodory-Pesin structures \cite{p97}. 

Let $\epsilon>0, s \geq 0, N\in \mathbb{N}, \mu \in \mathcal{M}(X)$ and $\delta \in (0,1)$. 
Put
$$\Lambda_{N,\epsilon}^s(\mu, \delta)=\inf\sum_{i\in I}\limits  e^{-n_i s},$$
where the infimum  is taken over all  finite or countable covers $\{B_{n_i}(x_i,e^{-n_i\epsilon})\}_{i\in I}$ so  that  $\mu (\cup_{i\in I}B_{n_i}(x_i, e^{-n_i\epsilon}))> 1-\delta$  with $n_i \geq N,  x_i\in X$.

Let $\Lambda_{\epsilon}^s(\mu, \delta)=\lim\limits_{N\to \infty}\Lambda_{N,\epsilon}^s(\mu, \delta).$ There is   a critical value  of parameter $s$ for $\Lambda_{\epsilon}^s(\mu, \delta)$  jumping from $\infty$ to $0$. We   define the critical value as
\begin{align*}
	\Lambda_{\epsilon}(\mu, \delta)=\inf\{s:\Lambda_{\epsilon}^s(\mu, \delta)=0\}=\sup\{s:\Lambda_{\epsilon}^s(\mu, \delta)=\infty\}.
\end{align*}
Let  $h_{\mu}^{\widetilde{K}}(T,\epsilon)=\lim_{\delta \to 0 }\limits \Lambda_{\epsilon}(\mu, \delta)$. 

\begin{df}
The   neutralized  Katok's entropy of  $\mu$ is  defined by 
\begin{align*}
h_{\mu}^{\widetilde{K}}(T)=\lim_{\epsilon \to 0}h_{\mu}^{\widetilde{K}}(T,\epsilon).
\end{align*}

\end{df}

We show that  the    lower  neutralized Brin-Katok local entropy  is a lower bound of  the   neutralized  Katok's entropy.

\begin{prop}\label{prop 2.5}
Let  $\mu \in \mathcal{M}(X)$.  Then for  every $\epsilon >0$,   one has
$$\underline{h}_{\mu}^{\widetilde{BK}}(T, \frac{\epsilon}{2})\leq h_{\mu}^{\widetilde{K}}(T,\epsilon)$$
and hence  $h_{\mu}^{\widetilde{BK}}(T)\leq h_{\mu}^{\widetilde{K}}(T)$.
\end{prop}
\begin{proof}
Assume that  $\underline{h}_{\mu}^{\widetilde{BK}}(T, \frac{\epsilon}{2})>0$.  We define 
$$E_N=\{x\in X:  \mu (B_n(x, e^{-\frac{n\epsilon}{2}}))<e^{-ns}~ \text{for all}~n\geq N\}.$$
 Let  $s<\underline{h}_{\mu}^{\widetilde{BK}}(T, \frac{\epsilon}{2})$.  Then there  exists $N_0$ with  $e^{\frac{N_0}{2}\epsilon}>2$ so that $\mu(E_{N_0})>0$. Fix  $\delta_0=\frac{1}{2}\mu(E_{N_0})>0$. Let $\{B_{n_i}(x_i,e^{-n_i\epsilon})\}_{i\in I}$ be a  finite or countable cover  so  that  $\mu (\cup_{i\in I}B_{n_i}(x_i, e^{-n_i\epsilon}))> 1-\delta_0$  with $n_i \geq N_0,  x_i\in X$. Then $\mu(E_{N_0}\cap  \cup_{i\in I}B_{n_i}(x_i, e^{-n_i\epsilon}))\geq \frac{1}{2}\mu(E_{N_0})>0$. Denote by $I_1=\{i\in I: E_{N_0}\cap  B_{n_i}(x_i, e^{-n_i\epsilon}) \not =\emptyset\}$. For every $i \in I_1$, we choose $y_i \in E_{N_0}\cap  B_{n_i}(x_i, e^{-n_i\epsilon})$ such that $$E_{N_0}\cap  B_{n_i}(x_i, e^{-n_i\epsilon})\subset  B_{n_i}(y_i, 2e^{-n_i\epsilon})\subset B_{n_i}(y_i, e^{-\frac{n_i}{2}\epsilon}).$$
Hence,
$$\sum_{i\in I}e^{-n_is}\geq \sum_{i\in I_1}e^{-n_is} \geq \sum_{i\in I_1}  \mu(B_{n_i}(y_i, e^{-\frac{n_i\epsilon}{2}}))\geq \frac{\mu(E_{N_0})}{2}>0.$$
This  implies that $\Lambda_{\epsilon}^s(\mu, \delta_0)\geq\Lambda_{N_0,\epsilon}^s(\mu, \delta_0)>0$ and hence $\Lambda_{\epsilon}(\mu, \delta_0)\geq s.$ Consequently, $ h_{\mu}^{\widetilde{K}}(T,\epsilon) \geq s$. Letting $s \to  \underline{h}_{\mu}^{\widetilde{BK}}(T, \frac{\epsilon}{2})$  this finishes the proof.
\end{proof}
\section{Proof of Theorem 1.1}

In this section, we prove the Theorem \ref{thm 1.1}.

Inspired by the idea of  the geometric measure theory \cite{m95} and the work of Feng and Huang \cite{fh12},  we first  introduce an equivalent definition  for  neutralized  Bowen topological entropy, called \emph{neutralized  weighted Bowen topological entropy}, which  allows us  to define a  positive bounded linear functional  $L$ on continuous function space of   $X$.  Then \emph{Riesz representation theorem}  can be applied  to  produce  a Borel probability measure $\mu \in \mathcal{M}(X)$ with $\mu(K)=1$ so that 
$M_{\epsilon}(K) \leq \underline{h}_{\mu}^{\widetilde{BK}}(T,2\epsilon).$

Let $f: X\rightarrow \mathbb{R}$  be a  bounded  real-valued  function on $X$. Let $s \geq 0$, $N\in \mathbb{N},$ and $\epsilon >0$. Define 
$$W_{N, \epsilon}^s(f)=\inf\sum_{i\in I}c_i e^{-n_is},$$
where the infimum is  taken over   all finite or countable families $\{(B_{n_i}(x_i, e^{-n_i \epsilon}),c_i)\}_{i \in I}$  with $0<c_i<\infty$, $x_i \in X$, $n_i \geq N$ so that   
$$\sum_{i\in I}\limits c_i \chi_{B_{n_i}(x_i, e^{-n_i \epsilon})}\geq f,$$
where $\chi_{A}$ denotes the characteristic function of $A$.

Let  $Z\subset X$ be a non-empty subset. We set $W_{N, \epsilon}^s(Z):=W_{N, \epsilon}^s(\chi_Z)$. Let 
$W_{ \epsilon}^s(Z)=\lim_{N \to \infty } W_{N, \epsilon}^s(Z).$
There is a critical value of $s$ so that  the quantity  $W_{ \epsilon}^s(Z)$ jumps from $\infty$ to $0$. We define such  critical value as 
\begin{align*}
	W_{ \epsilon}(Z):=\inf\{s : W_{ \epsilon}^s(Z)=0\}=\sup\{s : W_{ \epsilon}^s(Z)=\infty\}.
\end{align*}
\begin{df}
Given non-empty subset  $Z\subset X$, the   neutralized  weighted Bowen topological entropy of  of $T$  on the set $Z$  is  defined by 
\begin{align*}
h_{top}^{\widetilde{WB}}(T, Z)=\lim_{\epsilon \to 0}W_{ \epsilon}(Z).
\end{align*}

\end{df}  

The following  lemma is a variant of the classical $5r$-covering lemma  developed in \cite[Theorem 2.1]{m95}, which is useful for shortening the proof  of the equivalence of   neutralized   Bowen topological  entropy and  neutralized  weighted Bowen topological entropy.

\begin{lem}\label{lem 3.2}\rm{\cite[Lemma 6.3]{w21}}
Let $(X,d)$ be a compact metric space.  Let $r>0$ and  $\mathcal{B}=\{B(x_i,r)\}_{i\in I}$ be a family of  open balls of $X$. Define
$$I(i):=\{j\in I:B(x_j,r)\cap B(x_i,r)\not= \emptyset\}.$$
Then there exists a finite  index subset $J \subset I$ so that for any $i,j\in J$ with $i\not= j$, $I(i)\cap I(j)=\emptyset$ and
$$\cup_{i\in I}B(x_i,r)\subset \cup_{j\in J}B(x_j,5r).$$
\end{lem}

\begin{prop}\label{prop 3.3}
Let $Z$ be a non-empty subset of $X$ and $\epsilon>0, s>0,\theta >0$. Then for sufficiently large $N$,  we have $$M_{N,\frac{\epsilon}{2}}^{s+\theta}(Z) \leq  W_{N, \epsilon}^s(Z)\leq M_{N, \epsilon}^s(Z).$$
Consequently,  $M_{\frac{\epsilon}{2}}^{s}(Z) \leq  W_{ \epsilon}^s(Z)\leq M_{\epsilon}^s(Z)$ for any $s >0$  and hence  $$h_{top}^{\widetilde{B}}(T, Z)=h_{top}^{\widetilde{WB}}(T, Z).$$
\end{prop}

\begin{proof}
The  inequality  $W_{N, \epsilon}^s(Z)\leq M_{N, \epsilon}^s(Z)$  follows by the definitions.   We show  $M_{N,\frac{\epsilon}{2}}^{s+\theta}(Z) \leq  W_{N, \epsilon}^s(Z)$ by slightly modifying the proof of \cite[Proposition 6.4]{w21}. 

Let $N$ be a sufficiently large  integer so that $\sum_{n\geq N}\frac{1}{n^2}<1$, $e^{\frac{n\epsilon}{2}}>5$ and $\frac{n^2}{e^{n\theta}} <1$  hold for all $n\geq N$.  Let  $\{(B_{n_i}(x_i, e^{-n_i \epsilon}),c_i)\}_{i \in I}$  with $0<c_i<\infty$, $x_i \in X$, $n_i \geq N$ be a  finite or countable family satisfying  
$\sum_{i\in I}\limits c_i \chi_{B_{n_i}(x_i, e^{-n_i \epsilon})}\geq \chi_Z.$
Define $I_n=\{i\in I:n_i=n\}$, where $n\geq N$.  Let $t >0$ and $n\geq N$. We  define $$Z_{n,t}=\{z\in Z:  \sum_{i\in I_n}\limits c_i \chi_{B_{n}(x_i, e^{-n\epsilon})}(z)> t \}$$
and 
$$I_n^t=\{i\in I_n: B_{n}(x_i, e^{-n\epsilon})\cap Z_{n,t}\not=\emptyset\}.$$
Then  $Z_{n,t}\subset \cup_{i\in I_n^t} B_{n}(x_i, e^{-n\epsilon}).$ Let $\mathcal{B}=\{B_{n}(x_i, e^{-n\epsilon})\}_{i \in I_n^t}$. By Lemma \ref{lem 3.2},  there exists  a finite index  subset $J\subset I_n^t$ such  that

$$\cup_{i\in I_n^t} B_{n}(x_i, e^{-n\epsilon})\subset \cup_{j\in J} B_{n}(x_j, 5e^{-n\epsilon})\subset  \cup_{j\in J} B_{n}(x_j, e^{-\frac{n\epsilon}{2}})$$
and for any $i,j\in J$ with $i\not= j$, $I_n^t(i)\cap I_n^t(j)=\emptyset$, where  $I_n^t(i)=\{j\in I_n^t: B_{n}(x_j, e^{-n\epsilon})\cap B_{n}(x_i, e^{-n\epsilon})\not=\emptyset\}$. For each $j\in J$, we choose $y_j\in  B_{n}(x_j, e^{-n\epsilon})\cap Z_{n,t} $.  Then $\sum_{i\in I_n^t}\limits c_i \chi_{B_{n}(x_i, e^{-n\epsilon})}(y_j)>  t$ and hence $ \sum_{i\in I_n^t(j)}\limits c_i  > t$.   Summing  for $j \in J$ one  has
$$\#J< \frac{1}{t}\sum_{j\in J} \sum_{i\in I_n^t(j)}\limits c_i \leq  \frac{1}{t}\sum_{i\in I_n^t}\limits c_i. $$
It follows that $M_{N,\frac{\epsilon}{2}}^{s+\theta}(Z_{n,t}) \leq \#J\cdot e^{-n(s+\theta)}\leq  \frac{1}{n^2t} \sum_{i\in I_n}\limits c_i e^{-ns}$. Note that for any $t\in (0,1)$, $Z=\cup_{n\geq N}Z_{n,\frac{1}{n^2}t}$. So  $$M_{N,\frac{\epsilon}{2}}^{s+\theta}(Z) \leq \sum_{n\geq N}M_{N,\frac{\epsilon}{2}}^{s+\theta}(Z_{n,\frac{1}{n^2}t})\leq \frac{1}{t} \sum_{i\in I}\limits c_i e^{-n_is}.$$
Letting $t\to 1$ this  implies  the desired result.

\end{proof}

\begin{lem}[Frostman's lemma]\label{lem 3.4}
Let $K$ be a non-empty compact subset of $X$ and $s\geq 0$, $N\in \mathbb{N}$, $\epsilon >0$. Set $c:=W_{N,\epsilon}^s(K)>0$. Then  there exists a Borel probability measure $\mu$ on $X$ such that $\mu(K)=1$ and  for any $x\in X$, $n\geq N$,
$$\mu(B_{n}(x, e^{-n\epsilon}))\leq \frac{1}{c}e^{-ns}.$$
\end{lem}

\begin{proof}
The proof of \cite[Lemma 3.4]{fh12} also  works  for  Lemma 3.4  if   the ball  $B_n(x,\epsilon)$ is replaced by $B_{n}(x, e^{-n\epsilon})$. So we omit the detail proof.
\end{proof}

\begin{proof}[Proof of Theorem \ref{thm 1.1}]
Notice that for every $\mu \in \mathcal{M}(X)$ with $\mu(K)=1$ and $\epsilon >0$,  $h_{\mu}^{\widetilde{K}}(T,\epsilon)\leq M_{\epsilon}(K)$. Together with  the Proposition \ref{prop 2.5}, we have 
\begin{align*}
&\lim_{\epsilon \to 0}\sup\{\underline{h}_{\mu}^{\widetilde{BK}}(T,\epsilon):\mu\in \mathcal{M}(X),\mu(K)=1\}\\
\leq&\lim\limits_{\epsilon \to 0}\sup\{\underline{h}_{\mu}^{\widetilde{K}}(T,\epsilon):\mu\in \mathcal{M}(X),\mu(K)=1\}\leq h_{top}^{\widetilde{B}}(T,K).
\end{align*}

On the other hand, fix $\epsilon >0$ and assume that $M_{\epsilon}(K)>0$.  Let $s<M_{\epsilon}(K)$. Then  one has $W_{2\epsilon}^s(K)>0$ by Proposition  \ref{prop 3.3}. So  there exists $N$ such that   $c:=W_{N,2\epsilon}^s(K)>0$. By Lemma \ref{lem 3.4},
there exists a Borel probability measure $\mu$ on $X$ such that $\mu(K)=1$ and  for any $x\in X$, $n\geq N$,
$$\mu(B_{n}(x, e^{-2n\epsilon}))\leq \frac{1}{c}e^{-ns}.$$
This yields  that $\underline{h}_{\mu}^{\widetilde{BK}}(T,2\epsilon)\geq s$. Letting $s \to M_{\epsilon}(K) $, we obtain that 
$$M_{\epsilon}(K) \leq \underline{h}_{\mu}^{\widetilde{BK}}(T,2\epsilon)\leq \sup\{\underline{h}_{\mu}^{\widetilde{BK}}(T,2\epsilon):\mu\in \mathcal{M}(X),\mu(K)=1\}.$$
This completes  the proof.
\end{proof}

Finally, we provide an example for Theorem \ref{thm 1.1}.

\begin{ex}
	Under the setting of Example \ref{ex 2.7}, $h_{top}^{\widetilde{B}}(\sigma,\Sigma_N)=\log N$.  Let $(\Sigma_N,\sigma)$ be $(\frac{1}{N},...,\frac{1}{N})$-Bernoulli shift and $\mu$ be the corresponding product measure. 
	
	We show $\underline{h}_{\mu}^{\widetilde{BK}}(\sigma)=\log N$.  Using Brin-Katok formula,  we know that  $\underline{h}_{\mu}^{\widetilde{BK}}(\sigma)\geq \underline{h}_{\mu}^{{BK}}(\sigma)=\log N$.  Notice that 
	$$\underline{h}_{\mu}^{\widetilde{BK}}(\sigma)=\lim\limits_{k \to \infty}\int \liminf_{m\to \infty}\limits -\frac{1}{mk}\log \mu ({\overline{B}}_{mk}(x, e^{-mk\cdot\frac{1}{k}}))d\mu.$$   Fix $k,m\geq 1$.  For every $x\in [j_0\cdots j_{m(k+1)}]$, one has  $$[j_0\cdots j_{m(k+1)}]\subset \overline{B}_{mk}(x, e^{-mk\cdot\frac{1}{k}}).$$ This shows that 
	\begin{align*}
		&\int \liminf_{m\to \infty}\limits -\frac{\log \mu ({\overline{B}}_{mk}(x, e^{-mk\cdot\frac{1}{k}}))}{mk}d\mu\\
		\leq &\sum_{j_0,...,j_{m(k+1)}\in \{0,...,N-1\}} \int_{[j_0\cdots j_{m(k+1)}]} \liminf_{m\to \infty}\limits -\frac{\log {(\frac{1}{N})^{m(k+1)+1}}}{mk}d\mu\\
		=&(1+\frac{1}{k})\log N.
	\end{align*}
	Hence, $\underline{h}_{\mu}^{\widetilde{BK}}(\sigma)=\log N$. 
	
	Consequently,  we obtain that  $h_{top}^{\widetilde{B}}(\sigma,\Sigma_N)=\underline{h}_{\mu}^{\widetilde{BK}}(\sigma)=\log N.$
\end{ex}

\section*{Final remarks and open questions}
Notice that  for  lower  neutralized Brin-Katok local   entropy   and   the   neutralized Katok entropy $$\underline{h}_{\mu}^{\widetilde{BK}}(T)=\inf_{\epsilon > 0}\underline{h}_{\mu}^{\widetilde{BK}}(T,\epsilon), {h}_{\mu}^{\widetilde{K}}(T)=\inf_{\epsilon > 0}{h}_{\mu}^{\widetilde{K}}(T,\epsilon),$$ 
while $\underline{h}_{\mu}^{{BK}}(T)=\sup_{\epsilon > 0}\underline{h}_{\mu}^{{BK}}(T,\epsilon), {h}_{\mu}^{{K}}(T)=\sup_{\epsilon > 0}{h}_{\mu}^{{K}}(T,\epsilon).$  Using this fact,   one can exchange the orders  of  operations 
\begin{align*}
	&\lim_{\epsilon \to 0}\sup_{\mu\in \mathcal{M}(X),\mu(K)=1}\underline{h}_{\mu}^{{BK}}(T,\epsilon),\\
	&\lim_{\epsilon \to 0}\sup_{\mu\in \mathcal{M}(X),\mu(K)=1}\underline{h}_{\mu}^{{K}}(T,\epsilon).
\end{align*}
which is in general not true for    lower  neutralized Brin-Katok local   entropy   and   the   neutralized Katok entropy.  It is the  obstacle coming from the definition of neutralized entropies-like quantities that prevents  us to  establish variational principle for neutralized  Bowen topological entropy  whose form is  more close to  Feng and Huang's work \rm{\cite[Theorem 1.2]{fh12}}.  Interested reader can turn to \cite[Theorem 1.4]{w21} and \cite[Theorem 1.4]{ycz22} for an  analogous problem arising in metric mean dimension theory.

We  end up this paper with several questions. 
\begin{itemize}[leftmargin=7pt]
\item [1.] Does there exist TDS such that
$${ h}_{top}^B(T,Z)< {\widetilde h}_{top}^B(T,Z), \text{~or~} { h}_{top}(T,Z)< {\widetilde h}_{top}(T,Z)$$
for some subset $Z$?

\item  [2.] Does there exist TDS such that
$$ {\widetilde h}_{top}^B(T,X)< {\widetilde h}_{top}(T,X)\text{?}$$
\item  [3.]Given a TDS $(X,T)$ and  an invariant measure $\mu \in M(X,T)$, does $\underline{h}_{\mu}^{\widetilde{BK}}(T)=h_{\mu}(T) \text{?}$
\end{itemize}

The first question will help us  make clear  whether the two types of entropies  have  essential differences.
For Bowen topological entropy,  one has $h_{top}^B(T,X)=h_{top}^B(T,X)$ by \cite[Proposition 1]{b73}. In Example  \ref{ex 2.7}, the strict inequality can occur for subsets of $X$,   so we wonder if  it is true for  the  neutralized  Bowen topological entropy.
For the  third  question, Ovadia and Rodriguez-Hertz \cite[Theorem 3.4]{orh23} has   given a partial answer for  closed Riemannian manifold $M$ with  $f\in  \text{Diff}^{1+\beta}(M)$, however the general case is still open.

\section*{Acknowledgement} 

\noindent  We would like to thank  the anonymous referees for abundant valuable  comments that  greatly improved the previous  manuscript. The first author is  supported by Postgraduate Research $\&$ Practice Innovation Program of Jiangsu Province (No. KYCX23$\_$1665). The second author is  supported by the
National Natural Science Foundation of China (No.12071222). The third author is  supported by the
National Natural Science Foundation of China
(No. 11971236).  The work was also funded by the Priority Academic Program Development of Jiangsu Higher Education Institutions.  We would like to express our gratitude to Tianyuan Mathematical Center in Southwest China(No.11826102), Sichuan University and Southwest Jiaotong University for their support and hospitality.

\end{document}